\documentclass[final,3p]{elsarticle}

\usepackage{amssymb}
\usepackage{amsthm}

\usepackage{amsmath}
\usepackage{mathrsfs}
\usepackage[color=Purple!50!white]{todonotes}
\newtheorem{thm}{Theorem}[section]
\newtheorem*{thm*}{Theorem}
\newtheorem{lemma}[thm]{Lemma}

\newtheorem{conjecture}[thm]{Conjecture}
\newtheorem{qu}[thm]{Question}
\theoremstyle{definition}

\newcommand{\flo}[1]{\lfloor #1 \rfloor}
\newcommand{\cei}[1]{\lceil #1 \rceil}

\tikzset{global scale/.style={
		scale=#1,
		every node/.append style={scale=#1}
	}
}

\begin{document}

\begin{frontmatter}



\title{Proof of a Conjecture on Online Ramsey Numbers of Stars versus Paths}


\author[inst1,inst2]{Ruyu Song}

\affiliation[inst1]{organization={School of Mathematical Sciences, Hebei Normal University},
            city={Shijiazhuang},
            postcode={050024}, 
            state={Hebei},
            country={China}}

\author[inst1,inst2]{Sha Wang}
\author[inst1,inst2]{Yanbo Zhang\corref{cor1}}
\cortext[cor1]{Corresponding author.}
\ead{ybzhang@hebtu.edu.cn}
\affiliation[inst2]{organization={Hebei International Joint Research Center for Mathematics and Interdisciplinary Science},
            city={Shijiazhuang},
            postcode={050024}, 
            state={Hebei},
            country={China}}

\begin{abstract}
Given two graphs $G$ and $H$, the online Ramsey number $\tilde{r}(G,H)$ is defined to be the minimum number of rounds that Builder can always guarantee a win in the following $(G, H)$-online Ramsey game between Builder and Painter. Starting from an infinite set of isolated vertices, in each round Builder draws an edge between some two vertices, and Painter immediately colors it red or blue. Builder’s goal is to force either a red copy of $G$ or a blue copy of $H$ in as few rounds as possible, while Painter's goal is to delay it for as many rounds as possible. Let $K_{1,3}$ denote a star with three edges and $P_{\ell}$ a path with $\ell$ vertices. Latip and Tan conjectured that $\tilde{r}(K_{1,3}, P_{\ell})=(3/2+o(1))\ell$ [Bull.\ Malays.\ Math.\ Sci.\ Soc.\ 44 (2021) 3511--3521]. We show that $\tilde{r}(K_{1,3}, P_{\ell})=\lfloor 3\ell/2 \rfloor$ for $\ell\ge 2$, which verifies the conjecture in a stronger form.
\end{abstract}




\end{frontmatter}


\section{Introduction}

We are concerned with a game-theoretic notion called online Ramsey number in this paper. Given two graphs $G$ and $H$, the \emph{$(G, H)$-online Ramsey game} is a combinatorial game played on an infinite set of vertices between two players, Builder and Painter. Starting from an edgeless graph, in each round Builder draws an edge between two nonadjacent vertices, and Painter colors it red or blue immediately. Builder’s goal is to force either a red copy of $G$ or a blue copy of $H$ in as few rounds as possible, while Painter's goal is to delay it for as many rounds as possible. The \emph{online Ramsey number} $\tilde{r}(G,H)$ is the smallest number of rounds needed to create either a red copy of $G$ or a blue copy of $H$, assuming that both Builder and Painter play optimally.

The online Ramsey number was introduced by Beck \cite{Beck1993Achievement} and owes its name to Kurek and Ruci\'{n}ski \cite{Kurek2005Two}. It can be viewed as the online version of size Ramsey number. Recall that the Ramsey number $r(G,H)$ and the size Ramsey number $\hat{r}(G,H)$ are the smallest number of vertices and edges, respectively, in a graph $F$ such that for any red-blue edge coloring of $F$, there is either a red copy of $G$ or a blue copy of $H$. Clearly $\tilde{r}(G,H)\le \hat{r}(G,H)\le \binom{r(G,H)}{2}$.

In the classical case where both $G$ and $H$ are complete graphs, we write $m$, $n$ instead of $K_m$, $K_n$ for simplicity of notation. It is well known that the Ramsey number $r(n,n)$ is between $2^{n/2}$ and $2^{2n}$. Both bounds have seen no exponential improvements in decades \cite{Gowers2000two}. For size Ramsey numbers, a result attributed to Chv\'{a}tal shows that $\hat{r}(m,n)=\binom{r(m,n)}{2}$. While for online Ramsey numbers, Conlon \cite{Conlon2010line} showed that for infinitely many $n$,  $\tilde{r}(n,n)\le 1.001^{-n}\hat{r}(n,n)$, which is an exponential improvement for infinite numbers. For the lower bound, Beck \cite{Beck1993Achievement} described an elegant proof of $\tilde{r}(n,n)\ge r(n,n)/2$ which was found by Alon. Recently Conlon, Fox, Grinshpun, and He \cite{Conlon2019Online} proved that $\tilde{r}(n,n)\ge 2^{(2-\sqrt{2})n+O(1)}$, 
which is an exponential improvement to the lower bound. The basic conjecture, attributed by Kurek and Ruci\'{n}ski \cite{Kurek2005Two}, is to show that $\tilde{r}(n,n)=o(\hat{r}(n,n))$. This conjecture is at present far from being solved. For the exact values, only two nontrivial ones were obtained: $\tilde{r}(3,3)=8$ by Kurek and Ruci\'{n}ski \cite{Kurek2005Two}, and $\tilde{r}(3,4)=17$ by Pra\l{}at \cite{Pralat2008$R3417$}.

For sparse graphs, the online Ramsey numbers involving paths, stars, trees, and cycles have been studied \cite{Adamski2021Online,Adamski2022Online,Cyman2014note,Cyman2015line,Dybizbanski2020line,Dzido2021Note,Grytczuk2008line,Latip2021Note,Pralat2008note,Pralat2012note,Song2022Online}. In general, however, exact results are rare. Most results are upper or lower bounds, between which there is always a challenging gap. We are interested in exploring the following problem.
\begin{qu}
	Determine the exact expressions of some online Ramsey functions $f(n)=\tilde{r}(G,H_n)$, where $G$ is a small fixed graph, and $H_n$ is a graph from a class of sparse graphs such as paths, stars, and cycles.
\end{qu}
Let $P_n$ and $C_n$ denote a path of order $n$ and a cycle of order $n$, respectively. In 2015, Cyman, Dzido, Lapinskas, and Lo \cite{Cyman2015line} obtained the exact online Ramsey numbers of $P_3$ versus all paths and all cycles, which are $\tilde{r}(P_3,P_n)=\lceil 5(n-1)/4 \rceil$ for $n\ge 3$ and $\tilde{r}(P_3,C_n)=\lceil 5n/4 \rceil$ for $n\ge 5$. In the same paper they showed that $2n-2\le \tilde{r}(C_4,P_n)\le 4n-8$ for $n\ge 5$. Dybizba\'{n}ski, Dzido, and Zakrzewska \cite{Dybizbanski2020line} improved the upper bound of $\tilde{r}(C_4,P_n)$ from $4n-8$ to $3n-5$. Finally, Adamski and Bednarska-Bzd\c{e}ga \cite{Adamski2022Online} closed the gap by proving that $\tilde{r}(C_4,P_n)=2n-2$ for $n\ge 8$. Their proof involves a technical inductive argument, in which the base case is computer-assisted. In 2022, Song and Zhang \cite{Song2022Online} showed that $\tilde{r}(P_3,nK_2)=\lceil 3n/2 \rceil$ for $n\ge 2$; $\tilde{r}(P_4,nK_2)=\lceil 9n/5 \rceil$ for $n\ge 2$; $\tilde{r}(mK_2, P_n)=n+2m-4$ for $m=2,3$ and $n\ge 5$. Here, $nK_2$ denotes a matching with $n$ edges.

In 2021, Latip and Tan \cite{Latip2021Note} studied the function $\tilde{r}(K_{1,3}, P_\ell)$ and obtained that $\cei{3(\ell-1)/2}\le \tilde{r}(K_{1,3}, P_\ell)\le 5\ell/3+2$. Here we use the clearer notation $K_{1,3}$ rather than $S_3$ to denote a star with three edges, because $S_n$ may denote a star with $n$ vertices or a star with $n$ edges in different literatures. Latip and Tan believed that the lower bound is close to the exact expression and posed the following conjecture.

\begin{conjecture}\cite{Latip2021Note}
	$\tilde{r}(K_{1,3}, P_{\ell})=(3/2+o(1))\ell$.
\end{conjecture}

In this paper, we determine the exact value of $\tilde{r}(K_{1,3}, P_{\ell})$ for all $\ell\ge 2$, which verifies the above conjecture in a stronger form.

\begin{thm}
	\label{thm:mainresult}
	$\tilde{r}(K_{1,3}, P_{\ell})=\flo{3\ell/2}$ for $\ell\ge 2$.
\end{thm}

The remainder of this paper is organized as follows. Section \ref{sec2} shows the lower bound of $\tilde{r}(K_{1,3}, P_{\ell})$. We will prove the upper bound by induction for $\ell\ge 7$. Thus we divide the proof of the upper bound into two sections: Section \ref{sec3} for $2\le \ell\le 6$ and Section \ref{sec4} for $\ell\ge 7$.

\section{The lower bound}\label{sec2}

During the whole game Painter uses a blocking strategy: she always colors an edge red unless doing so would create either a red $K_{1,3}$ or a red cycle $C_k$ for $3\le k\le \flo{\ell/2}$. More specifically, let $R_i$ be the graph induced by all red edges before the $i$th round, and let $e_i$ be the edge chosen by Builder in his $i$th move. If $R_i+e_i$ contains a $K_{1,3}$ or a cycle whose length is at most $\flo{\ell/2}$, then Painter colors $e_i$ blue. Otherwise she colors $e_i$ red.

Let $G$, $R$, and $B$ be the graphs induced by all edges, all red edges, and all blue edges before the $\flo{3\ell/2}$th round, respectively. We will prove that $G$ contains neither a red $K_{1,3}$ nor a blue $P_\ell$. Hence Builder can not win the game in $\flo{3\ell/2}-1$ rounds, and the lower bound follows.

If $R$ has at least $\flo{\ell/2}+1$ edges, then $B$ has at most  $\flo{3\ell/2}-1-(\flo{\ell/2}+1)$ edges, which is $\ell-2$. In this way, $B$ can not contain a path $P_\ell$. So we assume that $R$ contains no cycle with length at least $\flo{\ell/2}+1$. Combining it with the Painter's strategy, we see that $R$ contains no cycle, and has maximum degree at most two. Consequently, $R$ consists of a disjoint union of paths. Assume that the number of components in $R$ is $s$. Let $X$ be the set of vertices which have degree two, and $\{y_{2i-1}, y_{2i}\}$ the ends of the $i$th path for $1\le i\le s$. Thus, $R$ has $|X|+s$ edges, which is at most $\flo{\ell/2}$.

Let $P$ be a longest blue path in $G$. Since each vertex is incident to at most two edges of a path, $P$ contains at most $2|X|$ edges incident to $X$. Recall that each blue edge is forced to appear, to avoid either a red $K_{1,3}$ or a red cycle. In other words, each blue edge is either incident to a vertex of $X$, or $y_{2i-1}y_{2i}$ for some $i$ with $1\le i\le s$. Therefore, the total number of edges in $P$ is at most $2|X|+s$. Since $|E(R)|+|E(P)|\le |E(G)|$, we have $$3|X|+2s\le \flo{3\ell/2}-1.$$ Suppose to the contrary that there is a blue path of order $\ell$, then
\begin{equation}\label{equ1}
	2|X|+s\ge \ell-1.
\end{equation}

If $s\ge 2$, then we have
\begin{align*}
	\flo{3\ell/2}-1&\ge 3|X|+2s\\
	&\ge 3(2|X|+s)/2+1\\
	&\ge 3(\ell-1)/2+1\\
	&=(3\ell-1)/2,
\end{align*}
which is a contradiction. So we have $s=1$. The inequality (\ref{equ1}) implies that $2|X|\ge \ell-2$. It follows that $|X|+s\ge \ell/2$. Since $R$ has $|X|+s$ edges, we see that $\ell/2\le |X|+s\le \flo{\ell/2}$. Thus, $\ell$ is even, and $R$ is a red path with $\ell/2$ edges (and $\ell/2+1$ vertices). If Builder joins the two end vertices of this path, he creates a cycle of length $\ell/2+1$. Painter will color this edge red by her strategy. Thus, the blue edges can only be forced by the vertices with degree two in $R$. Hence $G$ contains only $\ell-2$ blue edges, which contradicts our assumption that there is a blue $P_\ell$. Thus we have the lower bound.

\section{The upper bound for $2\le \ell\le 6$}\label{sec3}
We start with the following simple lemma.

\begin{lemma}\label{simplelemma}
	If there is a blue $P_k$ in the online Ramsey game, then in the next three rounds, there is either a blue $P_{k+1}$ or a red $K_{1,3}$.
\end{lemma}
\begin{proof}
	Let $v$ be an end vertex of the blue $P_k$. Builder joins $v$ to three new vertices in the next three rounds. If none of the three edges is blue, we have a red $K_{1,3}$. Otherwise, we have a blue $P_{k+1}$.
\end{proof}

For $\ell=2,3$, Builder draws a star with $\ell+1$ edges. Then there is either a red $K_{1,3}$ or a blue $P_\ell$. Hence $\tilde{r}(K_{1,3}, P_\ell)\le \ell+1= \flo{3\ell/2}$ for $\ell=2,3$. For $\ell=4$, Builder joins $v_0$ to $v_1,v_2,\ldots$ one by one, until a blue $K_{1,2}$ appears. This blue $K_{1,2}$ can be obtained in at most four rounds, since otherwise a red $K_{1,3}$ shows up and our proof is done. If the blue $K_{1,2}$ is obtained in three rounds, by Lemma \ref{simplelemma}, our proof is done. If the blue $K_{1,2}$ shows up in the fourth round, we assume that $v_0v_1, v_0v_2$ are red, and $v_0v_3, v_0v_4$ are blue. Builder chooses $v_1v_3$ and $v_1v_4$ in the next two moves. If neither of them is blue, then there is a red $K_{1,3}$. Otherwise, there is a blue $P_4$. Thus $\tilde{r}(K_{1,3}, P_4)\le 6$.

For $5\le \ell\le 6$, Builder first draws a path $P_4$, denoted by $v_1v_2v_3v_4$. It has five color patterns up to symmetry: $bbb$, $bbr$, $brb$, $brr$, $rrr$. In fact, the path may have another color pattern $rbr$, which is avoided as follows. If the first two edges $v_1v_2$ and $v_2v_3$ are colored blue and red respectively, then Builder joins $v_4$ to $v_3$. On the other hand, if $v_1v_2$ is red and $v_2v_3$ is blue, then Builder joins $v_4$ to $v_1$. Thus $rbr$ can not show up in the first three rounds.

We assume that there is no red $K_{1,3}$ in $\flo{3\ell/2}$ rounds, since otherwise our proof is done. If $v_1v_2v_3v_4$ has color pattern $bbb$, by Lemma \ref{simplelemma}, $\tilde{r}(K_{1,3}, P_\ell)\le \flo{3\ell/2}$ for $\ell=5,6$. If it has color pattern $brb$, Builder joins a new vertex $v_5$ to $v_2$ and $v_3$ in the next two moves. If both $v_2v_5$ and $v_3v_5$ are blue, then $v_1v_2v_5v_3v_4$ is a blue $P_5$. By Lemma \ref{simplelemma}, Builder forces a blue $P_6$ in the next three rounds. If both $v_2v_5$ and $v_3v_5$ are red, then for $\ell=5$, Builder chooses two edges $v_1v_5$ and $v_4v_5$, both of which are forced to be blue. Hence $v_2v_1v_5v_4v_3$ is a blue $P_5$. For $\ell=6$, Builder chooses three edges $v_2v_6$, $v_5v_6$, and $v_4v_5$, all of which are forced to be blue. Hence $v_1v_2v_6v_5v_4v_3$ is a blue $P_6$. If $v_2v_5$ and $v_3v_5$ have different colors, say, $v_2v_5$ is red and $v_3v_5$ is blue, then Builder joins $v_2$ to $v_4$ and hence $v_1v_2v_4v_3v_5$ is a blue $P_5$. By Lemma \ref{simplelemma}, Builder forces a blue $P_6$ in the next three rounds.

If $v_1v_2v_3v_4$ has color pattern $bbr$, Builder joins $v_3$ to a new vertex $v_5$. If $v_3v_5$ is blue, then Builder chooses $v_1v_4$ and $v_4v_5$ in the next two moves. At least one of the two edges is blue. Thus, we have a blue $P_5$ in six rounds. By Lemma \ref{simplelemma}, Builder forces a blue $P_6$ in the next three rounds. If $v_3v_5$ is red, then Builder joins $v_4$ to two new vertices $v_6$ and $v_7$. At least one of the two new edges is blue. If exactly one of them is blue, say, $v_4v_6$ is blue and $v_4v_7$ is red, then Builder chooses $v_3v_6$ for $\ell=5$, and then joins $v_4$ to a new vertex $v_8$ for $\ell=6$. As a result, $v_1v_2v_3v_6v_4$ is a blue $P_5$ and $v_1v_2v_3v_6v_4v_8$ is a blue $P_6$. If both $v_4v_6$ and $v_4v_7$ are blue, then Builder chooses $v_3v_6$, which has to be blue. Hence we obtain a blue $P_6$ (and also a blue $P_5$) in seven rounds, which is $v_1v_2v_3v_6v_4v_7$.

\begin{figure}[htp]
	\centering
	\begin{tikzpicture}[global scale=1.2]
		\begin{scope}[line width=1pt, blue]
			\draw (1,1) to (2,1);
			\draw (1,1) to [bend left=45](3,1);
			\draw (2,1) to [bend right=45](4,1);
			\draw (3.5,2) to (3,1);
		\end{scope}      	  		
		
		\begin{scope}[line width=1pt, red, dashed]
			\draw (2,1) to (3,1);
			\draw (3,1) to (4,1);
		\end{scope}
		\fill (3.5,2) circle (3pt);
		\fill (1,1) circle (3pt);
		\fill (2,1) circle (3pt);
		\fill (3,1) circle (3pt);
		\fill (4,1) circle (3pt);
		\node at (3,0.3) {\textcircled{\small{4}}};
		\node at (2,1.6) {5};
		\node at (3.1,1.6) {6};
	\end{tikzpicture}
	\hspace{20pt}
	\begin{tikzpicture}[global scale=1.2]
		\begin{scope}[line width=1pt, blue]
			\draw (1,1) to (2,1);
			\draw (2,1) to (2.5,2);
			\draw (3,1) to (2.5,2);
			\draw (3.5,2) to (3,1);
			\draw (3.5,2) to (4,1);
		\end{scope}      	  		
		
		\begin{scope}[line width=1pt, red, dashed]
			\draw (2,1) to (3,1);
			\draw (3,1) to (4,1);
			\draw (2,1) to [bend right=45](4,1);
		\end{scope}
		\fill (2.5,2) circle (3pt);
		\fill (3.5,2) circle (3pt);
		\fill (1,1) circle (3pt);
		\fill (2,1) circle (3pt);
		\fill (3,1) circle (3pt);
		\fill (4,1) circle (3pt);
		\node at (3,0.3) {\textcircled{\small{4}}};
		\node at (2.1,1.6) {5};
		\node at (2.9,1.6) {6};
		\node at (3.15,1.6) {7};
		\node at (3.9,1.6) {8};
	\end{tikzpicture}
	\hspace{20pt}
	\begin{tikzpicture}[global scale=1.2]
		\begin{scope}[line width=1pt, blue]
			\draw (2,1) to (1.5,2);
			\draw (2,1) to (2.5,2);
			\draw (3,1) to (2.5,2);
			\draw (3.5,2) to (3,1);
		\end{scope}
		\begin{scope}[line width=1pt, dotted]
			\draw (1.5,2) to (4,1);
			\draw (3.5,2) to (4,1);
		\end{scope}    	  		
		
		\begin{scope}[line width=1pt, red, dashed]
			\draw (1,1) to (2,1);
			\draw (2,1) to (3,1);
			\draw (3,1) to (4,1);
		\end{scope}
		\fill (1.5,2) circle (3pt);
		\fill (2.5,2) circle (3pt);
		\fill (3.5,2) circle (3pt);
		\fill (1,1) circle (3pt);
		\fill (2,1) circle (3pt);
		\fill (3,1) circle (3pt);
		\fill (4,1) circle (3pt);
		\node at (3,0.15) {};
		\node at (1.85,1.6) {4};
		\node at (2.1,1.6) {5};
		\node at (2.9,1.6) {6};
		\node at (3.15,1.6) {7};
	\end{tikzpicture}
	\caption{The color patterns $brr$ and $rrr$ of the first three edges.}
	\label{5and6}
\end{figure}
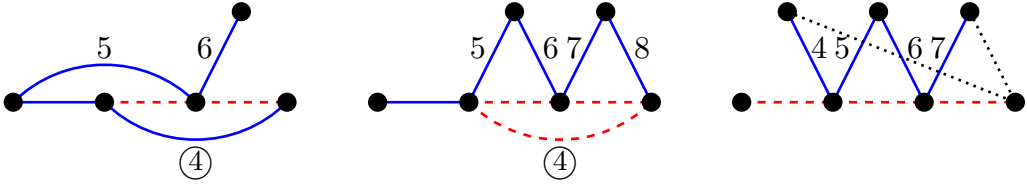

For the other two color patterns $brr$ and $rrr$, we illustrate Builder's strategy as in Figure \ref{5and6}. In the pattern $brr$, we use a circled number to denote that Painter has a choice in that move. Since there is a blue $P_5$ in six rounds in the left graph, it follows that a blue $P_6$ can be forced in nine rounds. In the pattern $rrr$, at least one of the black dotted edges is blue. To summarize, Builder can always force a blue $P_5$ in seven rounds and a blue $P_6$ in nine rounds.

\section{The upper bound for $\ell \ge 7$}\label{sec4}

In this section we show that $\tilde{r}(K_{1,3}, P_\ell)\le \flo{3\ell/2}$ for $\ell\ge 7$. Assume that Painter always avoids a red $K_{1,3}$. When an edge is `forced' to be blue in the following, it means that there is a red $K_{1,3}$ if the edge is colored red. We shall find a blue $P_\ell$ in $\flo{3\ell/2}$ rounds. In the first two moves Builder draws a star $K_{1,2}$, which has three possible color patterns up to symmetry: $bb$, $br$, $rr$. For the last two patterns, Builder then joins the center of $K_{1,2}$ to a new vertex. Thus, there are three cases in total: the first two rounds form a blue path of order three; the first three rounds form a star with two edges blue and one edge red; the first three rounds form a star with two edges red and one edge blue. In most cases, the Builder's strategy is as follows. First he creates a small graph $H$. Then for some integer $k$ with $4\le k\le \ell-2$, Builder forces a blue path $P_{\ell-k}$ that is vertex-disjoint with $H$ by induction. Finally, he combines $H$ and $P_{\ell-k}$ to a blue path $P_\ell$.

\begin{figure}[htp]
	\centering
	\begin{tikzpicture}
		\begin{scope}[line width=1pt, blue]
			\draw (1,0.5) to (2,0.5);
			\draw (2,0.5) to (3,0.5);
		\end{scope}      	  		
		
		\fill (1,0.5) circle (3pt);
		\fill (2,0.5) circle (3pt);
		\fill (3,0.5) circle (3pt);
		
		\node at (2,-0.5) {Case 1};
	\end{tikzpicture}
	\hspace{40pt}
	\begin{tikzpicture}
		\begin{scope}[line width=1pt, blue]
			\draw (2,1) to (3,1);
			\draw (1,1) to (2,1);
		\end{scope}
		
		\begin{scope}[line width=1pt, red, dashed]
			\draw (2,1) to (2,0);
		\end{scope}
		
		\fill (1,1) circle (3pt);
		\fill (2,1) circle (3pt);
		\fill (2,0) circle (3pt);
		\fill (3,1) circle (3pt);
		\node at (2,-0.5) {Case 2};
		
	\end{tikzpicture}
	\hspace{40pt}
	\begin{tikzpicture}
		\begin{scope}[line width=1pt, blue]
			\draw (2,1) to (2,0);
		\end{scope}
		
		\begin{scope}[line width=1pt, red, dashed]
			\draw (1,1) to (2,1);
			\draw (2,1) to (3,1);
		\end{scope}
		
		\fill (1,1) circle (3pt);
		\fill (2,1) circle (3pt);
		\fill (2,0) circle (3pt);
		\fill (3,1) circle (3pt);
		\node at (2,-0.5) {Case 3};
	\end{tikzpicture}
	\caption{We distinguish three cases.}
	\label{fig:cases}
\end{figure}
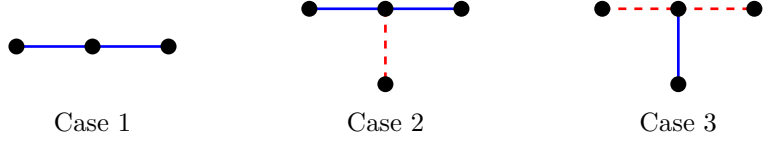

\textbf{Case 1.} The first two rounds form a blue path of order three.

Assume that the blue path is $v_1v_2v_3$. Builder extends it to a longer path $v_1v_2v_3v_4v_5$ in the next two moves. We distinguish three subcases by the colors of $v_3v_4$ and $v_4v_5$.

\textbf{Subcase 1.1.} The edge $v_4v_5$ is red.

\begin{figure}[htp]
	\centering
	\begin{tikzpicture}[global scale=0.98]
		\begin{scope}[line width=1pt, blue]
			\draw (1.5,3) to [bend left] (4.5,3);
			\draw (1,1) to (2,1);
			\draw (2,1) to (3,1);
			\draw (1,1) to [bend left=45](4,1);
			\draw (3,1) to [bend right=60](5,1);
			\draw (4.5,3) to (4,1);
		\end{scope}      	  		
		
		\begin{scope}[line width=1pt, red, dashed]
			\draw (3,1) to (4,1);
			\draw (4,1) to (5,1);
		\end{scope}
		\fill (1.5,3) circle (3pt);
		\fill (4.5,3) circle (3pt);
		\fill (1,1) circle (3pt);
		\fill (2,1) circle (3pt);
		\fill (3,1) circle (3pt);
		\fill (4,1) circle (3pt);
		\fill (5,1) circle (3pt);
		\node at (1.5,2.7) {$x$};
		\node at (4.6,2.7) {$y$};
		\node at (1,0.7) {$v_1$};
		\node at (2,0.7) {$v_2$};
		\node at (3,0.7) {$v_3$};
		\node at (4,0.7) {$v_4$};
		\node at (5.1,0.7) {$v_5$};
		\node at (3,3.65) {$P_{\ell-5}$};
		
	\end{tikzpicture}
	\begin{tikzpicture}[global scale=0.98]
		\begin{scope}[line width=1pt, blue]
			\draw (1.5,3) to [bend left] (4.5,3);
			\draw (1,1) to (2,1);
			\draw (2,1) to (3,1);
			\draw (1,1) to [bend left=45](4,1);
			\draw (4,1) to (4.5,1.5);
			\draw (5,1) to (4.5,1.5);
			\draw (4.5,3) to (5,1);
			
		\end{scope}      	  		
		
		\begin{scope}[line width=1pt, red, dashed]
			\draw (3,1) to (4,1);
			\draw (4,1) to (5,1);
			\draw (3,1) to [bend right=60](5,1);
		\end{scope}
		\fill (1.5,3) circle (3pt);
		\fill (4.5,3) circle (3pt);
		\fill (1,1) circle (3pt);
		\fill (2,1) circle (3pt);
		\fill (3,1) circle (3pt);
		\fill (4,1) circle (3pt);
		\fill (5,1) circle (3pt);
		\fill (4.5,1.5) circle (3pt);
		\node at (1.5,2.7) {$x$};
		\node at (4.75,2.75) {$y$};
		\node at (1,0.7) {$v_1$};
		\node at (2,0.7) {$v_2$};
		\node at (3,0.7) {$v_3$};
		\node at (4,0.7) {$v_4$};
		\node at (5.1,0.7) {$v_5$};
		\node at (4.5,1.8) {$v_6$};
		\node at (3,3.65) {$P_{\ell-6}$};
	\end{tikzpicture}
	\begin{tikzpicture}[global scale=0.98]
		\begin{scope}[line width=1pt, blue]
			\draw (1.5,3) to [bend left] (4.5,3);
			\draw (1,1) to (2,1);
			\draw (2,1) to (3,1);
			\draw (4.5,3) to (4,1);
			\draw (3,1) to (4,1);
		\end{scope}
		
		\begin{scope}[line width=1pt]
			\draw (4,1) to (1.5,3);
			\draw (4,1) to (4.5,3);
			
		\end{scope} 		
		
		\begin{scope}[line width=1pt, red, dashed]
			\draw (4,1) to (5,1);
		\end{scope}
		\begin{scope}[line width=1pt, white]
			\draw (3,1) to [bend right=60](5,1);
		\end{scope}
		\fill (1.5,3) circle (3pt);
		\fill (4.5,3) circle (3pt);
		\fill (1,1) circle (3pt);
		\fill (2,1) circle (3pt);
		\fill (3,1) circle (3pt);
		\fill (4,1) circle (3pt);
		\fill (5,1) circle (3pt);
		\node at (1.5,2.7) {$x$};
		\node at (4.6,2.7) {$y$};
		\node at (1,0.7) {$v_1$};
		\node at (2,0.7) {$v_2$};
		\node at (3,0.7) {$v_3$};
		\node at (4,0.7) {$v_4$};
		\node at (5,0.7) {$v_5$};
		\node at (3,3.65) {$P_{\ell-4}$};
	\end{tikzpicture}
	\caption{Force a $P_\ell$ in Subcase 1.1.}
\end{figure}
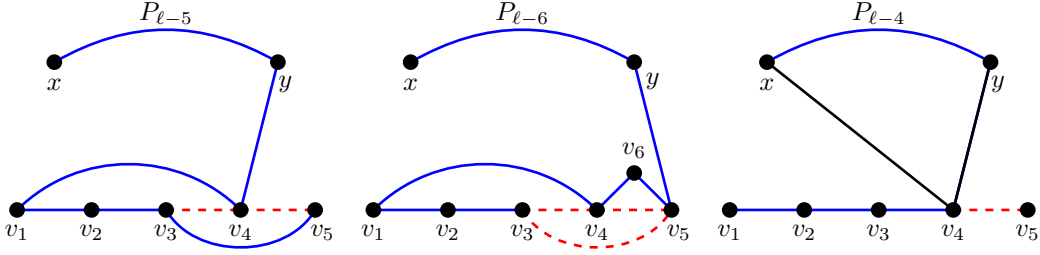

If $v_3v_4$ is red, Builder then chooses $v_3v_5$. If $v_3v_5$ is blue, Builder forces a blue $P_{\ell-5}$ in $\flo{3\ell/2}-7$ rounds by induction, whose end vertices are denoted by $x$ and $y$. Next Builder draws two edges $yv_4$ and $v_4v_1$, both of which are forced to be blue. Thus, there is a blue path $xP_{\ell-5}yv_4v_1v_2v_3v_5$ of order $\ell$ in $\flo{3\ell/2}$ rounds. If $v_3v_5$ is red and $\ell=7$, Builder draws four edges $yv_5$, $v_5v_6$, $v_6v_4$, and $v_4v_1$, all of which are blue. Here, $v_6$ and $y$ are two new vertices. Thus, $yv_5v_6v_4v_1v_2v_3$ is a blue $P_7$. If $v_3v_5$ is red and $\ell\ge 8$, Builder forces a blue $P_{\ell-6}$ with two ends $x$ and $y$ in $\flo{3\ell/2}-9$ rounds by induction. Next Builder draws four edges $yv_5$, $v_5v_6$, $v_6v_4$, and $v_4v_1$. Here, $v_6$ is a new vertex, and the four edges are forced to be blue. Thus, there is a blue path $xP_{\ell-6}yv_5v_6v_4v_1v_2v_3$ of order $\ell$ in $\flo{3\ell/2}$ rounds.

If $v_3v_4$ is blue, Builder then forces a blue $P_{\ell-4}$ with two ends $x$ and $y$ in $\flo{3\ell/2}-6$ rounds by induction. Next Builder draws two edges $xv_4$ and $yv_4$, at least one of which is blue. Without loss of generality, assume that $yv_4$ is blue. Thus, there is a blue path $xP_{\ell-4}yv_4v_3v_2v_1$ of order $\ell$ in $\flo{3\ell/2}$ rounds.

\textbf{Subcase 1.2.} The edge $v_3v_4$ is red and $v_4v_5$ is blue.

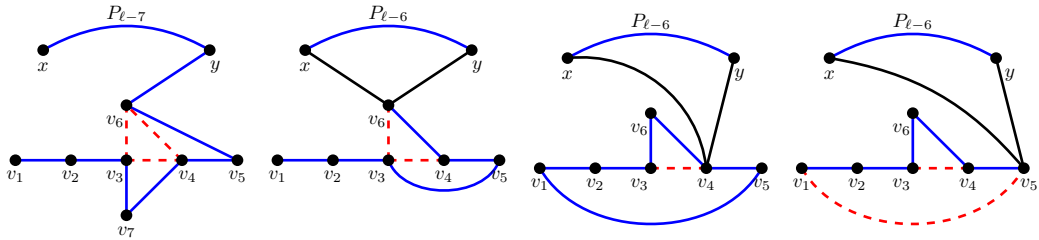
\begin{figure}[htp]
	\centering
	\begin{tikzpicture}[global scale=0.73]
		\begin{scope}[line width=1pt, blue]
			\draw (1.5,3) to [bend left] (4.5,3);
			\draw (1,1) to (2,1);
			\draw (2,1) to (3,1);
			\draw (4,1) to (5,1);
			\draw (3,1) to (3,0);
			\draw (4,1) to (3,0);
			\draw (5,1) to (3,2);
			\draw (3,2) to (4.5,3);
		\end{scope}      	  		
		
		\begin{scope}[line width=1pt, red, dashed]
			\draw (3,1) to (4,1);
			\draw (3,1) to (3,2);
			\draw (4,1) to (3,2);
		\end{scope}
		\fill (1.5,3) circle (3pt);
		\fill (4.5,3) circle (3pt);
		\fill (1,1) circle (3pt);
		\fill (2,1) circle (3pt);
		\fill (3,1) circle (3pt);
		\fill (4,1) circle (3pt);
		\fill (5,1) circle (3pt);
		\fill (3,2) circle (3pt);
		\fill (3,0) circle (3pt);
		\node at (1.5,2.7) {$x$};
		\node at (4.6,2.7) {$y$};
		\node at (1,0.7) {$v_1$};
		\node at (2,0.7) {$v_2$};
		\node at (2.8,0.7) {$v_3$};
		\node at (4.1,0.7) {$v_4$};
		\node at (5,0.7) {$v_5$};
		\node at (2.8,1.7) {$v_6$};
		\node at (3,-0.3) {$v_7$};
		\node at (3,3.65) {$P_{\ell-7}$};
	\end{tikzpicture}
	\begin{tikzpicture}[global scale=0.73]
		\begin{scope}[line width=1pt, blue]
			\draw (1.5,3) to [bend left] (4.5,3);
			\draw (1,1) to (2,1);
			\draw (2,1) to (3,1);
			\draw (4,1) to (5,1);
			\draw (4,1) to (3,2);
			\draw (3,1) to [bend right=70] (5,1);
		\end{scope}
		\begin{scope}[line width=1pt, white]
			\node at (3,-0.3) {$v_7$};
		\end{scope}    	  		
		\begin{scope}[line width=1pt]
			\draw (3,2) to (1.5,3);
			\draw (3,2) to (4.5,3);
		\end{scope}
		\begin{scope}[line width=1pt, red, dashed]
			\draw (3,1) to (4,1);
			\draw (3,1) to (3,2);
		\end{scope}
		\fill (1.5,3) circle (3pt);
		\fill (4.5,3) circle (3pt);
		\fill (1,1) circle (3pt);
		\fill (2,1) circle (3pt);
		\fill (3,1) circle (3pt);
		\fill (4,1) circle (3pt);
		\fill (5,1) circle (3pt);
		\fill (3,2) circle (3pt);
		\node at (1.5,2.7) {$x$};
		\node at (4.6,2.7) {$y$};
		\node at (1,0.7) {$v_1$};
		\node at (2,0.7) {$v_2$};
		\node at (2.8,0.7) {$v_3$};
		\node at (4,0.7) {$v_4$};
		\node at (5,0.7) {$v_5$};
		\node at (2.8,1.7) {$v_6$};
		\node at (3,3.65) {$P_{\ell-6}$};
	\end{tikzpicture}
	\begin{tikzpicture}[global scale=0.73]
		\begin{scope}[line width=1pt, blue]
			\draw (1.5,3) to [bend left] (4.5,3);
			\draw (1,1) to (2,1);
			\draw (2,1) to (3,1);
			\draw (4,1) to (5,1);
			\draw (3,1) to (3,2);
			\draw (4,1) to (3,2);
			\draw (1,1) to [bend right=60] (5,1);
		\end{scope}      	  		
		\begin{scope}[line width=1pt]
			\draw (4,1) to [bend right=45](1.5,3);
			\draw (4,1) to (4.5,3);
		\end{scope}
		\begin{scope}[line width=1pt, red, dashed]
			\draw (3,1) to (4,1);
		\end{scope}
		\fill (1.5,3) circle (3pt);
		\fill (4.5,3) circle (3pt);
		\fill (1,1) circle (3pt);
		\fill (2,1) circle (3pt);
		\fill (3,1) circle (3pt);
		\fill (4,1) circle (3pt);
		\fill (5,1) circle (3pt);
		\fill (3,2) circle (3pt);
		\node at (1.5,2.7) {$x$};
		\node at (4.6,2.7) {$y$};
		\node at (1,0.7) {$v_1$};
		\node at (2,0.7) {$v_2$};
		\node at (2.8,0.7) {$v_3$};
		\node at (4,0.7) {$v_4$};
		\node at (5,0.7) {$v_5$};
		\node at (2.8,1.7) {$v_6$};
		\node at (3,3.65) {$P_{\ell-6}$};
	\end{tikzpicture}
	\begin{tikzpicture}[global scale=0.73]
		\begin{scope}[line width=1pt, blue]
			\draw (1.5,3) to [bend left] (4.5,3);
			\draw (1,1) to (2,1);
			\draw (2,1) to (3,1);
			\draw (4,1) to (5,1);
			\draw (3,1) to (3,2);
			\draw (4,1) to (3,2);
		\end{scope}      	  		
		\begin{scope}[line width=1pt]
			\draw (5,1) to [bend right=20](1.5,3);
			\draw (5,1) to (4.5,3);
		\end{scope}
		\begin{scope}[line width=1pt, red, dashed]
			\draw (3,1) to (4,1);
			\draw (1,1) to [bend right=60] (5,1);
		\end{scope}
		\fill (1.5,3) circle (3pt);
		\fill (4.5,3) circle (3pt);
		\fill (1,1) circle (3pt);
		\fill (2,1) circle (3pt);
		\fill (3,1) circle (3pt);
		\fill (4,1) circle (3pt);
		\fill (5,1) circle (3pt);
		\fill (3,2) circle (3pt);
		\node at (1.5,2.7) {$x$};
		\node at (4.75,2.7) {$y$};
		\node at (1,0.7) {$v_1$};
		\node at (2,0.7) {$v_2$};
		\node at (2.8,0.7) {$v_3$};
		\node at (4,0.7) {$v_4$};
		\node at (5.1,0.7) {$v_5$};
		\node at (2.8,1.7) {$v_6$};
		\node at (3,3.65) {$P_{\ell-6}$};
	\end{tikzpicture}
	\caption{Force a $P_\ell$ in Subcase 1.2.}
\end{figure}

Builder joins both $v_3$ and $v_4$ to a new vertex $v_6$. If $v_3v_6$ and $v_4v_6$ are red, and $\ell=7$, then $v_1v_2v_3v_7v_4v_5v_6$ is a blue $P_7$, where $v_7$ is a new vertex. If $v_3v_6$ and $v_4v_6$ are red, and $\ell=8$, then $v_1v_2v_3v_7v_4v_5v_6y$ is a blue $P_8$, where $v_7$ and $y$ are two new vertices. If $v_3v_6$ and $v_4v_6$ are red, and $\ell\ge 9$, then Builder forces a blue $P_{\ell-7}$ in $\flo{3\ell/2}-10$ rounds by induction, whose end vertices are denoted by $x$ and $y$. It follows that $v_1v_2v_3v_7v_4v_5v_6yP_{\ell-7}x$ is a blue $P_\ell$, where $v_7$ is a new vertex. Thus, we obtain a blue $P_\ell$ in the required rounds.

If $v_3v_6$ is red and $v_4v_6$ is blue, and $\ell=7$, then Builder chooses $v_3v_5$, which has to be blue. He then joins $v_6$ to two new vertices $x$ and $y$. Either $v_6x$ or $v_6y$ is blue. Thus, we have a blue $P_7$, which is $v_1v_2v_3v_5v_4v_6x$ or $v_1v_2v_3v_5v_4v_6y$. If $v_3v_6$ is red and $v_4v_6$ is blue, and $\ell\ge 8$, then Builder forces a blue $P_{\ell-6}$ with two ends $x$ and $y$ in $\flo{3\ell/2}-9$ rounds by induction. Next Builder chooses $v_3v_5$, $xv_6$, and $yv_6$. The edge $v_3v_5$ has to be blue, and at least one of $xv_6$ and $yv_6$, say $yv_6$, is blue. As a result, $v_1v_2v_3v_5v_4v_6yP_{\ell-6}x$ is a blue $P_\ell$. We obtain a blue $P_\ell$ in $\flo{3\ell/2}$ rounds again. If $v_3v_6$ is blue and $v_4v_6$ is red, applying the same argument as above, we can obtain a blue $P_\ell$.

If both $v_3v_6$ and $v_4v_6$ are blue, and $\ell=7$, we have obtained a blue $P_6$ in six rounds, which is $v_1v_2v_3v_6v_4v_5$. Thus, it is easy to obtain a blue $P_7$ in nine rounds. If both $v_3v_6$ and $v_4v_6$ are blue, and $\ell\ge 8$, then Builder forces a blue $P_{\ell-6}$ with two ends $x$ and $y$ in $\flo{3\ell/2}-9$ rounds by induction. Next Builder chooses the edge $v_1v_5$. If $v_1v_5$ is blue, Builder joins $v_4$ to $x$ and $y$. At least one of $xv_4$ and $yv_4$, say, $yv_4$ is blue. Accordingly, $v_5v_1v_2v_3v_6v_4yP_{\ell-6}x$ is a blue $P_\ell$. Hence we obtain a blue $P_\ell$ in $\flo{3\ell/2}$ rounds. If $v_1v_5$ is red, Builder then joins $v_5$ to $x$ and $y$. At least one of $xv_5$ and $yv_5$, say, $yv_5$ is blue. Accordingly, $v_1v_2v_3v_6v_4v_5yP_{\ell-6}x$ is a blue $P_\ell$. We obtain a blue $P_\ell$ in the required rounds.

\textbf{Subcase 1.3.} Both edges $v_3v_4$ and $v_4v_5$ are blue.

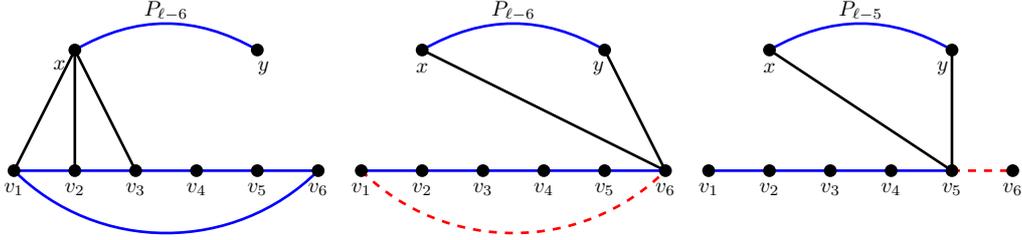
\begin{figure}[htp]
	\centering
	\begin{tikzpicture}[global scale=0.8]
		\begin{scope}[line width=1pt, blue]
			\draw (2,3) to [bend left] (5,3);
			\draw (1,1) to (2,1);
			\draw (2,1) to (3,1);
			\draw (3,1) to (4,1);
			\draw (4,1) to (5,1);
			\draw (5,1) to (6,1);
			\draw (1,1) to [bend right=45] (6,1);
		\end{scope}      	  		
		\begin{scope}[line width=1pt]
			\draw (1,1) to (2,3);
			\draw (2,1) to (2,3);
			\draw (3,1) to (2,3);
		\end{scope}
		
		\fill (2,3) circle (3pt);
		\fill (5,3) circle (3pt);
		\fill (1,1) circle (3pt);
		\fill (2,1) circle (3pt);
		\fill (3,1) circle (3pt);
		\fill (4,1) circle (3pt);
		\fill (5,1) circle (3pt);
		\fill (6,1) circle (3pt);
		\node at (1.75,2.75) {$x$};
		\node at (5.1,2.7) {$y$};
		\node at (1,0.7) {$v_1$};
		\node at (2,0.7) {$v_2$};
		\node at (3,0.7) {$v_3$};
		\node at (4,0.7) {$v_4$};
		\node at (5,0.7) {$v_5$};
		\node at (6,0.7) {$v_6$};
		\node at (3.5,3.65) {$P_{\ell-6}$};
	\end{tikzpicture}
	\begin{tikzpicture}[global scale=0.8]
		\begin{scope}[line width=1pt, blue]
			\draw (2,3) to [bend left] (5,3);
			\draw (1,1) to (2,1);
			\draw (2,1) to (3,1);
			\draw (3,1) to (4,1);
			\draw (4,1) to (5,1);
			\draw (5,1) to (6,1);
		\end{scope}      	  		
		\begin{scope}[line width=1pt]
			\draw (6,1) to (2,3);
			\draw (6,1) to (5,3);
		\end{scope}
		\begin{scope}[line width=1pt, red, dashed]
			\draw (1,1) to [bend right=45](6,1);
		\end{scope}
		\fill (2,3) circle (3pt);
		\fill (5,3) circle (3pt);
		\fill (1,1) circle (3pt);
		\fill (2,1) circle (3pt);
		\fill (3,1) circle (3pt);
		\fill (4,1) circle (3pt);
		\fill (5,1) circle (3pt);
		\fill (6,1) circle (3pt);
		\node at (2,2.7) {$x$};
		\node at (4.9,2.7) {$y$};
		\node at (1,0.7) {$v_1$};
		\node at (2,0.7) {$v_2$};
		\node at (3,0.7) {$v_3$};
		\node at (4,0.7) {$v_4$};
		\node at (5,0.7) {$v_5$};
		\node at (6,0.7) {$v_6$};
		\node at (3.5,3.65) {$P_{\ell-6}$};
	\end{tikzpicture}
	\begin{tikzpicture}[global scale=0.8]
		\begin{scope}[line width=1pt, blue]
			\draw (2,3) to [bend left] (5,3);
			\draw (1,1) to (2,1);
			\draw (2,1) to (3,1);
			\draw (3,1) to (4,1);
			\draw (4,1) to (5,1);
		\end{scope}      	  		
		\begin{scope}[line width=1pt]
			\draw (5,1) to (2,3);
			\draw (5,1) to (5,3);
		\end{scope}
		\begin{scope}[line width=1pt, white]
			\draw (1,1) to [bend right=45](6,1);
		\end{scope}
		\begin{scope}[line width=1pt, red, dashed]
			\draw (5,1) to (6,1);
		\end{scope}
		\fill (2,3) circle (3pt);
		\fill (5,3) circle (3pt);
		\fill (1,1) circle (3pt);
		\fill (2,1) circle (3pt);
		\fill (3,1) circle (3pt);
		\fill (4,1) circle (3pt);
		\fill (5,1) circle (3pt);
		\fill (6,1) circle (3pt);
		\node at (2,2.7) {$x$};
		\node at (4.85,2.7) {$y$};
		\node at (1,0.7) {$v_1$};
		\node at (2,0.7) {$v_2$};
		\node at (3,0.7) {$v_3$};
		\node at (4,0.7) {$v_4$};
		\node at (5,0.7) {$v_5$};
		\node at (6,0.7) {$v_6$};
		\node at (3.5,3.65) {$P_{\ell-5}$};
	\end{tikzpicture}
	\caption{Force a $P_\ell$ in Subcase 1.3.}
\end{figure}

Builder joins $v_5$ to a new vertex $v_6$ in the next move. If $v_5v_6$ is blue and $\ell=7$, Builder joins $v_6$ to three new vertices $v_7,v_8,v_9$. At least one of the three edges, say $v_6v_7$, is blue. Thus we obtain a blue $P_7$ in nine rounds. If $v_5v_6$ is blue and $\ell\ge 8$, Builder forces a blue $P_{\ell-6}$ in $\flo{3\ell/2}-9$ rounds by induction, whose end vertices are denoted by $x$ and $y$. Builder joins $v_1$ and $v_6$ in the next move. If $v_1v_6$ is blue, then all $v_i$'s for $1\le i\le 6$ form a cycle. Builder joins $x$ to $v_1,v_2,v_3$ respectively. At least one of $xv_1, xv_2, xv_3$ is blue. We may assume that $xv_1$ is blue without loss of generality. Accordingly, $v_2v_3v_4v_5v_6v_1xP_{\ell-6}y$ is a blue $P_\ell$. If $v_1v_6$ is red, Builder chooses $v_6x$ and $v_6y$ in the last two moves. We may assume that $v_6y$ is blue without loss of generality. It follows that $v_1v_2v_3v_4v_5v_6yP_{\ell-6}x$ is a blue $P_\ell$.

If $v_5v_6$ is red, Builder forces a blue $P_{\ell-5}$ with end vertices $x$ and $y$ in $\flo{3\ell/2}-7$ rounds by induction. Builder chooses $v_5x$ and $v_5y$ in the last two moves. To avoid a red $K_{1,3}$, we may assume that $v_5y$ is blue without loss of generality. It follows that $v_1v_2v_3v_4v_5yP_{\ell-5}x$ is a blue $P_\ell$. In all three cases, we obtain a blue $P_\ell$ in at most $\flo{3\ell/2}$ rounds.

\begin{figure}[htp]
	\centering
	\begin{tikzpicture}
		\begin{scope}[line width=1pt, blue]
			\draw (0,3) to [bend left] (4,3);
			\draw (1,1) to (2,1);
			\draw (2,1) to (3,1);
			\draw (1,1) to (2,2);
			\draw (0,3) to (2,2);
		\end{scope}
		
		\begin{scope}[line width=1pt, red, dashed]
			\draw (2,1) to (2,2);
			\draw (3,1) to (2,2);
		\end{scope}
		
		\fill (1,1) circle (3pt);
		\fill (2,1) circle (3pt);
		\fill (2,2) circle (3pt);
		\fill (3,1) circle (3pt);
		\fill (0,3) circle (3pt);
		\fill (4,3) circle (3pt);
		\node at (1,0.7) {$v_1$};
		\node at (2,0.7) {$v_2$};
		\node at (3,0.7) {$v_3$};
		\node at (2,2.3) {$v_4$};
		\node at (0,2.7) {$x$};
		\node at (4,2.7) {$y$};
		\node at (2,3.8) {$P_{\ell-4}$};
	\end{tikzpicture}
	\hspace{40pt}
	\begin{tikzpicture}
		\begin{scope}[line width=1pt, blue]
			\draw (0,3) to [bend left] (4,3);
			\draw (1,1) to (2,1);
			\draw (2,1) to (3,1);
			\draw (3,1) to (2,2);
		\end{scope}
		\begin{scope}[line width=1pt]
			\draw (0,3) to (2,2);
			\draw (4,3) to (2,2);
		\end{scope}
		\begin{scope}[line width=1pt, red, dashed]
			\draw (2,1) to (2,2);
		\end{scope}
		
		\fill (1,1) circle (3pt);
		\fill (2,1) circle (3pt);
		\fill (2,2) circle (3pt);
		\fill (3,1) circle (3pt);
		\fill (0,3) circle (3pt);
		\fill (4,3) circle (3pt);
		\node at (1,0.7) {$v_1$};
		\node at (2,0.7) {$v_2$};
		\node at (3,0.7) {$v_3$};
		\node at (2,2.3) {$v_4$};
		\node at (0,2.7) {$x$};
		\node at (4,2.7) {$y$};
		\node at (2,3.8) {$P_{\ell-4}$};
	\end{tikzpicture}
	\caption{Force a $P_\ell$ in Case 2.}
\end{figure}
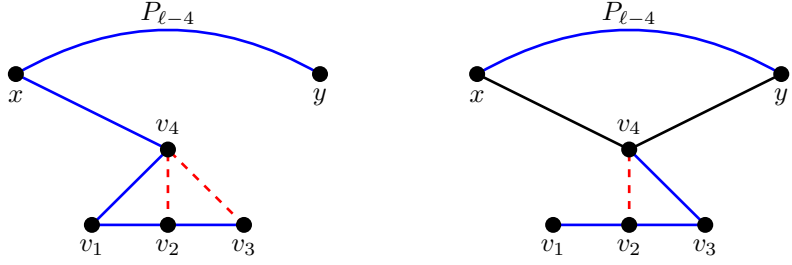

\textbf{Case 2.} A star with two edges blue and one edge red appears in the first three rounds.

Assume that the blue path is $v_1v_2v_3$, and the red edge is $v_2v_4$. Since $\tilde{r}(K_{1,3},P_{\ell-4})\le \flo{3\ell/2}-6$ by induction, Builder forces a blue $P_{\ell-4}$ in the next $\flo{3\ell/2}-6$ rounds, whose end vertices are denoted by $x$ and $y$. Then he chooses $v_3v_4$ in the next move. If $v_3v_4$ is red, he draws two edges $v_1v_4$ and $v_4x$, both of which are forced to be blue. Hence there is a blue path $v_3v_2v_1v_4xP_{\ell-4}y$ of order $\ell$ in $\flo{3\ell/2}$ rounds. If $v_3v_4$ is blue, he draws two edges $v_4x$ and $v_4y$. To avoid a red $K_{1,3}$, at least one of $v_4x$ and $v_4y$ is blue, say, $v_4x$ is blue. Thus, there is a blue path $v_1v_2v_3v_4xP_{\ell-4}y$ of order $\ell$ in $\flo{3\ell/2}$ rounds.

\textbf{Case 3.} A star with two edges red and one edge blue appears in the first three rounds.

Assume that the red path is $v_1v_2v_3$, and the blue edge is $v_2u_2$. We extend this red path in the following way. First Builder joins $v_3$ to two new vertices $u_3$ and $v_4$. If both $v_3u_3$ and $v_3v_4$ are blue, then we have found the required red path. If not, the two edges $v_3u_3$ and $v_3v_4$ have to be one red and one blue, since otherwise there is a red $K_{1,3}$, which contradicts our assumption. Without loss of generality, assume that $v_3u_3$ is blue and $v_3v_4$ is red. For each $i\ge 4$, if $v_{i-1}v_i$ is red, Builder joins $v_i$ to two new vertices $u_i$ and $v_{i+1}$. If both $v_iu_i$ and $v_iv_{i+1}$ are blue, then we stop the procedure. Otherwise, assume that $v_iu_i$ is blue and $v_iv_{i+1}$ is red. The procedure stops when either the red path has length $\flo{\ell/2}+1$, or both $v_tu_t$ and $v_tv_{t+1}$ are blue for some $t$ with $3\le t\le \flo{\ell/2}+1$.

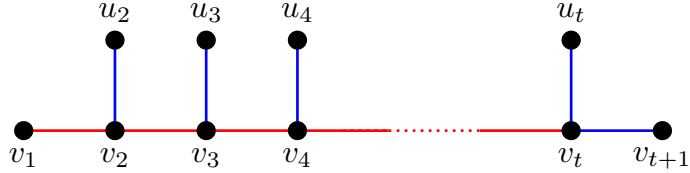
\begin{figure}[htp]
	\centering
	\begin{tikzpicture}[global scale=1.2]
		\begin{scope}[line width=1pt, blue]
			\draw (7,1) to (8,1);
			\draw (2,1) to (2,2);
			\draw (3,1) to (3,2);
			\draw (4,1) to (4,2);
			\draw (7,1) to (7,2);
		\end{scope}      	  		
		
		\begin{scope}[line width=1pt, red]
			\draw (1,1) to (2,1);
			\draw (2,1) to (3,1);
			\draw (3,1) to (4,1);
			\draw (4,1) to (5,1);
			\draw (6,1) to (7,1);
		\end{scope}
		\begin{scope}[line width=1pt, red, dotted]
			\draw (4.5,1) to (5,1);
			\draw (5,1) to (6,1);
		\end{scope}
		\fill (1,1) circle (3pt);
		\fill (2,1) circle (3pt);
		\fill (3,1) circle (3pt);
		\fill (4,1) circle (3pt);
		\fill (7,1) circle (3pt);
		\fill (8,1) circle (3pt);
		\fill (2,2) circle (3pt);
		\fill (3,2) circle (3pt);
		\fill (4,2) circle (3pt);
		\fill (7,2) circle (3pt);
		\node at (1,0.7) {$v_1$};
		\node at (2,0.7) {$v_2$};
		\node at (3,0.7) {$v_3$};
		\node at (4,0.7) {$v_4$};
		\node at (7,0.7) {$v_t$};
		\node at (2,2.3) {$u_2$};
		\node at (3,2.3) {$u_3$};
		\node at (4,2.3) {$u_4$};
		\node at (7,2.3) {$u_t$};
		\node at (8,0.7) {$v_{t+1}$};		
	\end{tikzpicture}
	
	\caption{The graph that Builder constructs in the first phase of Case 3.}
\end{figure}

If the red path has length $\flo{\ell/2}+1$, Builder joins $u_i$ to $v_{i+1}$ for each $i$ with $2\le i\le \flo{\ell/2}$. All edges $u_iv_{i+1}$ have to be blue to avoid a red $K_{1,3}$. Thus, if $\ell$ is even, $v_2u_2v_3u_3\cdots v_{\ell/2+1}u_{\ell/2+1}$ is a blue $P_\ell$, and if $\ell$ is odd, $u_1v_2u_2v_3u_3\cdots v_{(\ell-1)/2+1}u_{(\ell-1)/2+1}$ is a blue $P_\ell$, where $u_1$ is a new vertex. In both cases, Builder can force a blue $P_\ell$ in $\ell-1+\flo{\ell/2}+1$ rounds, which is $\flo{3\ell/2}$ rounds.

Now we consider the other case that there exists an integer $t$ with $3\le t\le \flo{\ell/2}+1$ such that both $v_tu_t$ and $v_tv_{t+1}$ are blue edges. Builder joins $v_i$ to $u_{i+1}$ for each $i$ with $2\le i\le t-1$. If $t=\flo{\ell/2}+1$ and $\ell$ is odd, then $u_2v_2u_3v_3\cdots u_tv_tv_{t+1}$ is a blue path of order $\ell$. If $t=\flo{\ell/2}+1$ and $\ell$ is even, then $u_2v_2u_3v_3\cdots u_tv_t$ is a blue path of order $\ell$. In both cases, a blue $P_\ell$ can be forced in $\flo{3\ell/2}$ rounds.

If $t=\flo{\ell/2}$, Builder joins $v_1$ to $v_{t+1}$. If $\ell$ is even and $v_1v_{t+1}$ is blue, then we can find a blue $P_\ell$, which is $u_2v_2u_3v_3\cdots u_tv_tv_{t+1}v_1$. If $\ell$ is even and $v_1v_{t+1}$ is red, then Builder chooses $v_1u_2$, and $v_1u_2v_2u_3v_3\cdots u_tv_tv_{t+1}$ is a blue $P_\ell$. If $\ell$ is odd and $v_1v_{t+1}$ is blue, then Builder draws two edges $v_1x$ and $v_1y$, at least one of which is blue, say, $v_1x$ is blue. It follows that $u_2v_2u_3v_3\cdots u_tv_tv_{t+1}v_1x$  is a blue $P_\ell$. If $\ell$ is odd and $v_1v_{t+1}$ is red, then Builder draws two edges $v_1x$ and $v_1u_2$, both of which are forced to be blue. Hence $xv_1u_2v_2u_3v_3\cdots u_tv_tv_{t+1}$ is a blue $P_\ell$. It is not difficult to check that the total number of rounds is at most $\flo{3\ell/2}$. Thus, we assume that $3\le t\le \flo{\ell/2}-1$. 

Since $\ell-2t\ge 2$, we have $\tilde{r}(K_{1,3}, P_{\ell-2t})\le \flo{3(\ell-2t)/2}$ by induction. Thus, Builder can force a blue $P_{\ell-2t}$ in the next $\flo{3(\ell-2t)/2}$ rounds. Assume that the end vertices of this $P_{\ell-2t}$ is $x$ and $y$. Builder joins $v_1$ to $v_{t+1}$. If $v_1v_{t+1}$ is red, Builder chooses $xv_1$ and $v_1u_2$, which are forced to be blue. Hence $yP_{\ell-2t}xv_1u_2v_2u_3v_3\cdots u_tv_tv_{t+1}$ is a blue $P_\ell$. If $v_1v_{t+1}$ is blue, Builder draws two edges $v_1x$ and $v_1y$, at least one of which is blue, say, $v_1x$. It follows that $u_2v_2u_3v_3\cdots u_tv_tv_{t+1}v_1xP_{\ell-2t}y$ is a blue $P_\ell$. Extending the path $xP_{\ell-2t}y$ to a blue $P_\ell$, we have used $2t$ blue edges and $t$ red edges. Accordingly, the total number of rounds is at most $\flo{3(\ell-2t)/2}+3t$, which is $\flo{3\ell/2}$. Therefore, $\tilde{r}(K_{1,3}, P_\ell)\le \flo{3\ell/2}$.



\end{document}